\documentclass[11pt,a4paper]{amsart}
\usepackage{graphicx} 
\usepackage{amsmath, amssymb,amsthm,mathtools}
\usepackage{tikz-cd}
\usepackage{todonotes}
\usepackage{svg}
\usepackage{xcolor}
\usepackage{hyperref}
\usepackage{comment}
\usepackage{geometry}

\setuptodonotes{size=\footnotesize}
\title{Stable homology of complex braid groups}
\author{Andrea Bianchi} 
\address{Dipartimento di Matematica, Università di Bologna, Italy} \email{andrea.bianchi37@unibo.it}
\author{Filippo Callegaro}
\address{Dipartimento di Matematica, Università di Pisa, Italy.} \email{callegaro@dm.unipi.it}
\thanks{The second author acknowledges INdAM/GNSAGA}
\author{Luigi Caputi}
\address{Dipartimento di Matematica, Università di Bologna, Italy}
\email{luigi.caputi@unibo.it }
\author{Paolo Salvatore}
\address{Dipartimento di Matematica, Università di Roma Tor Vergata, Italy}
\email{salvator@mat.uniroma2.it }
\thanks{The fourth author acknowledges the
MUR Excellence Project MatMod@TOV awarded to the Department of Mathematics, University
of Roma Tor Vergata, CUP E83C18000100006.}

\date{\today}

\newtheorem{theorem}{Theorem}[section]
\newtheorem{corollary}[theorem]{Corollary}

\newtheorem{lemma}[theorem]{Lemma}
\newtheorem{proposition}[theorem]{Proposition}

\newtheorem{athm}{Theorem}

\theoremstyle{definition}
\newtheorem{definition}[theorem]{Definition}

\newtheorem{remark}[theorem]{Remark}

\newtheorem{notation}[theorem]{Notation}

\newcommand{\bB}{\mathbf{B}} 
\newcommand{\D}{\mathcal{D}}
\newcommand{\A}{\mathcal{A}}
\newcommand{\DD}{\mathfrak{D}}
\renewcommand{\AA}{\mathfrak{A}}
\newcommand{\cS}{\mathcal{S}}
\newcommand{\fib}{\mathrm{fib}}
\newcommand{\map}{\mathrm{map}}
\newcommand{\Mod}{\mathbf{Mod}}
\newcommand{\bR}{\mathbb{R}}
\newcommand{\bC}{\mathbb{C}}
\newcommand{\cL}{\mathcal{L}}
\newcommand{\bZ}{\mathbb{Z}}
\newcommand{\bN}{\mathbb{N}}
\newcommand{\bz}{\mathbf{z}}
\newcommand{\Conf}{\mathrm{Conf}}
\newcommand{\CConf}{\mathfrak{Conf}}
\newcommand{\colim}{\mathrm{colim}}
\newcommand{\Tel}{\mathrm{Tel}}
\newcommand{\op}{\mathrm{op}}

\newcommand{\sslash}{\mathbin{/\mkern-6mu/}}
\begin{document}
\begin{abstract}
We compute the stable homology of complex braid groups of types $B(e,e,n)$ and $B(2e,e,n)$ for fixed $e\ge2$ and increasing $n$. This accounts for the stable homology of all infinite families of complex braid groups. We achieve this by explicitly computing a quillenization of their stable classifying spaces.  In particular, we provide a proof of an identification of the stable homology of Artin groups of type $D$ claimed by Fuchs in the '70s.
\end{abstract}
\maketitle

\section{Introduction}
This article is first concerned with the homology of Artin groups of type $D$. The subject is classical; for instance, Goryunov \cite{Gory78, Gory82} has  computed the entire homology of the group $D_n$ with integer coefficients for any $n$. Homological stability holds for the sequence of Artin groups of type $D$, and it is meaningful to ask for a concrete description of the quillenization (or plus construction) $\bB D_\infty^+$ of the classifying space of the colimit group $D_\infty:=\colim_{n\to\infty}D_n$.
In \cite[\S2.7, p.~34]{fuks74}, Fuchs claims the following:
\begin{quote}
The quillenization of the space $K(B_C, 1)$ is $(\Omega^2S^2)_0 \times \Omega S^2$, 
and
the quillenization of the space $K(B_D, 1)$ is the product of the space $(\Omega^2S^2)_0$ by the homotopy fiber of the map $S^3 \to S^3$ of degree $2$.
Both facts are proved on the basis of the theorem of Segal.
\end{quote}
Vershinin has written a detailed proof of the claim for the case $C$ in \cite[Thm.~11.1, p.~315]{Ver99}. For the case $D$, Fuchs claims the following equivalence of spaces:
\begin{equation}\label{eq:main} 
\bB D_\infty^+\simeq\Omega^2S^3\times\fib(S^3\xrightarrow{2}S^3).
\end{equation}
However, a direct and transparent derivation did not seem to be readily available to the authors. 
Therefore, an aim of the present paper is to provide a detailed proof of \eqref{eq:main}. More importantly, we place this result in a broader framework by computing the quillenization of a bigger family of configuration spaces 
closely related to the complex braid groups. This leads to a uniform description of the stabilized
homotopy types associated with the complex braid groups of type $B(e,e,n)$ and $B(2e,e,n)$, extending the picture suggested by Fuchs beyond the cases of type $B$ and $D$.
\begin{athm}
    \label{thm:0}
    For $e\ge1$ and $d \geq 2$ we have the following equivalences of spaces:
    \[
    \bB{B(e,e,\infty)}^+ \simeq \Omega^2S^3\times\fib(S^3\xrightarrow{e}S^3);\quad \quad\bB{B(de,e,\infty)}^+ \simeq \Omega^2S^3\times\Omega S^2.
    \]
\end{athm}

Theorem \ref{thm:0} follows directly from Theorem \ref{thm:A} and Theorem \ref{thm:B}. These will be phrased in terms of group completions of monoids and localizations of modules.

\begin{notation}
For an $E_1$-algebra in spaces $A$ and a right $E_1$-module $M$ over $A$ we denote by $G(A)$ the group completion of $A$, and by $L_A(M)$ the localization of $M$ with respect to all elements of $A$. We have an equivalence of group-like $E_1$-algebras (see \cite{BP72, May75, Quillen94})
\[
G(A)\simeq\Omega\bB A,
\]
where we abbreviate by $\bB A$ the bar construction $\bB(*,A,*)$ with respect to the terminal left and right $E_1$-modules over $A$; we similarly have
\[
L_A(M)\simeq \bB(M,A,G(A))\simeq \fib(\bB(M,A,*)\to \bB A),
\]
where $\fib$ denotes the (homotopy) fiber; see  Lemma \ref{lem:localization_via_fiber} for the above equivalences.
\end{notation}

We shall denote by $\A=\coprod_{n\ge0} \bB{A_{n-1}}$ the braid monoid, which we shall only regard as an $E_1$-monoid; both $\D(e)=\coprod_n \bB B(e,e,n)$ and $\D^*(e)=\coprod _n \bB B(2e,e,n)$ have natural structures of modules over $\A$: see Section \ref{sec:complex_braid_and_conf} for details. 

In \cite{Segal73} Segal proves that the group completion of $\A$ is $G(\A)\simeq \Omega^2S^2$, whereas the group-completion theorem \cite{McDuffSegal} implies the equivalence of spaces $G(\A)\simeq\bZ\times \bB A_\infty^+$. Note that after forgetting the $E_1$-algebra structure we have an equivalence of spaces $\Omega^2S^2\simeq\bZ \times \Omega^2S^3$, whence we also have $\bB A_\infty^+\simeq\Omega^2S^3$. Similarly, the quillenization of $\bB B(e,e,\infty)$ enters the equivalence of spaces 
$L_\A(\D(e))\simeq\bZ\times\bB B(e,e,\infty)^+$, see Lemma \ref{lem:localization_quillenization} for details.
\begin{athm}
\label{thm:A}
Let $e\ge2$. Then there is a homotopy equivalence of right $E_1$-modules over $G(\A)\simeq\Omega^2S^2$
\[
L_\A(\D(e))\simeq G(\A)\times \fib(S^3\xrightarrow{e}S^3),
\]
in which the action of $G(\A)$ on the right hand side is only on the first factor, and $S^3\xrightarrow{e}S^3$ denotes a degree $e$ self map of $S^3$.
\end{athm}
As an immediate application, we identify for any $e\geq 2$ the stable integral homology of complex braid groups of type $B(e,e,n)$ with the homology of the space $\Omega^2S^3\times\fib(S^3\xrightarrow{e}S^3)$; to the best of our knowledge, this is a new result in the literature.  
\begin{remark}
We recall (see, for example~\cite[Section~1.1]{brotolevi2000}) that we can compute the homology of $\fib(S^3 \stackrel{e}{\to}S^3)$ via the Serre spectral sequence associated with the homotopy fiber sequence $\fib(S^3 \stackrel{e}{\to}S^3)\to S^3\stackrel{e}{\to}S^3$: for $n\ge1$ we get that the $n$\textsuperscript{th} homology of $\fib(S^3 \stackrel{e}{\to}S^3)$ is isomorphic
to the cyclic group $\bZ/e$ if $n$ is even, and is trivial if $n$ is odd. The homology of $\bB A_\infty$ has been computed explicitly in \cite{fuks70, Vainshtein1978, cohen1976}.\end{remark}

As a side result, we also extend the statement of Fuchs about the quillenization of the space $K(B_C,1)$ to the general case of families of complex braid groups of type $B(2e,e,n)$ for fixed $e$.  It was already known (\cite[Prop. 4.24]{CalMar11}) that the homology of $B(2e,e,\infty)$ is equivalent to the homology of $B(2,1,\infty)$. Our second main result is the following;
\begin{athm}
\label{thm:B}
There is a homotopy equivalence of right $E_1$-modules over $G(\A)\simeq\Omega^2S^2$
\[
L_\A(\D^*(e)) \simeq G(\A)\times\Omega S^2\left<e\right>,
\]
in which the action of $G(\A)$ on the right hand side is only on the first factor, and $\Omega S^2\left<e\right>$ is the connected $e$-fold covering of the space $\Omega S^2$.  
\end{athm}
\begin{remark}
\label{rem:OmegaS2_splits}
We note that $\Omega S^2$ is homotopy equivalent to the product $S^1 \times \Omega S^3$. In fact, we can associate with the Hopf fibration the following homotopy fiber sequence:
\[
 \Omega S^3 \to \Omega S^2 \to S^1
\]
and the canonical map $S^1 \stackrel{\Omega \Sigma}{\longrightarrow} \Omega S^2$ gives a homotopy section. Using the structure of group-like $E_1$-algebra on $\Omega S^2$, we obtain an equivalence of spaces between $S^1 \times \Omega S^3$ and $\Omega S^2$.
As a consequence, all connected finite-degree coverings of $\Omega S^2$ are homotopy equivalent to $\Omega S^2$.
\end{remark}

\subsection*{Conventions}
Unless otherwise specified, all categories are assumed to be $\infty$-categories. The category $\cS$ denotes the category of spaces. We will use the word ``monoid'', ``group'' and ``module'' to refer to an $E_1$-algebra, a group-like $E_1$-algebra and an $E_1$-module (over a given $E_1$-algebra) in $\cS$, respectively.

\subsection*{Acknowledgments}

The authors would like to thank Rachael Boyd 
for her valuable contribution in the early stages of this project, and for her feedback on the first draft.
The authors also thank Ulrike Tillmann for useful discussions.
\section{Preliminaries and background}

\subsection{Complex  braid groups and configuration spaces}
\label{sec:complex_braid_and_conf}
Complex braid groups are the  generalizations of braid groups associated to arbitrary (finite) complex reflection groups. By a finite complex reflection group~$W$, we mean a finite subgroup of some $GL_n(\mathbb{C})$ generated by pseudo-reflections; that is, finite-order endomorphisms of $GL_n(\mathbb{C})$ which leave invariant some hyperplane in $\mathbb{C}^n$. 

The classification of irreducible finite complex reflection groups was first given by Shephard
and Todd~\cite{ShephardTodd}: there is an infinite family $G(de,e,n)$, where $d,e,n$ are positive integers, and 34 exceptional groups, denoted by $G_4,\dots, G_{37}$. The infinite family $G(de,e,n)$ includes the four infinite families $A_{n-1}$, $C_n$, $D_n$ and $I_2(e)$ of finite Coxeter groups, given by $G(1,1,n)$, $G(2,1,n)$, $G(2,2,n)$ and $G(e,e,2)$, respectively.  

We can associate to $W$ a braid group~$B(W)$. If we denote by $\mathcal{H}$ the corresponding collection of hyperplanes defined by $W$, the group $W$ acts freely on the hyperplane complement $X=\mathbb{C}^n\setminus\mathcal{H}$ (see \cite{Steinberg64}). The complex braid  group~$B(W)$ is defined as the fundamental group~$\pi_1(X/W)$ (\cite{BMR98}).  
It turns out that the complement $X$ is aspherical, and hence so is $X/W$. This result is due to many authors, in particular it is due to Fadell and Neuwirth \cite{FN62} for type $A_n$, to Brieskorn \cite{Brieskorn73} for type $C_n$ and $D_n$, to Nakamura \cite{Nak83} for the infinite series of complex reflection groups. As a consequence the quotient $X/W$ is a classifying space for the complex braid group $B(W)$.

In the follow-up, we shall  denote by $B(de,e,n)$ the braid group associated with $G(de,e,n)$.

While the groups $G(de,e,n)$ for distinct $d,e,n$ are pairwise non-isomorphic, for $d>1$ we have $B(2e,e,n) \simeq B(de,e,n)$ (see \cite[Prop.3.8]{BMR98}).

For completeness, we briefly recall the definition of the group $G(de,e,n)$. Let $d,e,n$ be positive integers and let $(z_1,\dots,z_n)$ be in $\mathbb{C}^n$. Following \cite{BMR98}, we denote by $G(de,e,n)$ the subgroup of $GL_n(\mathbb{C})$ consisting of the elements
\[
[\underline{a},\sigma]\colon z_j\mapsto a_j z_{\sigma(j)},
\]
where $\sigma\in\Sigma_n$ is a permutation, and $\underline{a}=(a_1,\dots,a_n)$, with  $a_j\in \mathbb{C}$ for all $j$,  satisfying $a_j^{de}=1$ and $(a_1\cdots a_r)^e=1$. 

For $d=1$, the complex braid group $B(e,e,n)$ has various presentations~\cite{BMR98, CorranPicantin}, but we shall not use them. 

We next introduce configuration spaces corresponding to the infinite family of the aforementioned complex reflection groups.
\begin{definition}
For $n\ge0$ we denote by $\Conf_n$ the space of unordered configurations of $n$ points in $\bC$:
\[
\Conf_n=\left\{(z_1,\dots,z_n)\in\bC\ |\ z_i\neq z_j\ \forall i\neq j\right\}/S_n.
\]
For $e\ge1$ and $n\ge0$ we denote by $\Conf_n(e)$ and $\Conf_n^*(e)$ the spaces
\[
\Conf_n(e)=\left\{(\{z_1,\dots,z_n\},y)\in\Conf_n\times\bC\ |\  y^e=\prod_{i=1}^nz_i\right\}
\]
and
\[
\Conf_n^*(e)=\left\{(\{z_1,\dots,z_n\},y)\in\Conf_n\times\bC^*\ |\  y^e=\prod_{i=1}^nz_i\right\}.
\]

Note that for $e=1$ we have $\Conf_n\cong\Conf_n(1)$.
We denote by $\A:=\coprod_{n\ge0}\Conf_n$, by $\D(e):=\coprod_{n\ge0}\Conf_n(e)$ and by $\D^*(e):=\coprod_{n\ge0}\Conf_n^*(e)$.

We have stabilization maps $\Conf_n(e)\to\Conf_{n+1}(e)$ and $\Conf^*_n(e)\to\Conf^*_{n+1}(e)$ (see Section \ref{sec:setting_modules} for details) and we denote by $\Conf_\infty(e)$ the space $\colim_{n\to\infty}\Conf_n(e)$ and by $\Conf^*_\infty(e)$ the space $\colim_{n\to\infty}\Conf_n^*(e)$.
\end{definition}

For $n\ge1$ the following hold:
\begin{itemize}
\item the space $\Conf_n$ is a classifying space for the Artin group of type $A_{n-1}$;
\item the space $\Conf_n^*(2)$ is a classifying space for the Artin group of type $C_n$;
\item the space $\Conf_n(2)$ is a classifying space for the Artin group of type $D_n$; 
\item 
the space $\Conf_n(e)$ is a classifying space for the complex braid group of type 
$B(e,e,n)$ and the space $\Conf_n^*(e)$ is a classifying space for the complex braid group of type 
$B(2e,e,n)$ according to the notation in \cite{BMR98};
\item 
homological stability holds for the sequence of complex braid groups of type $B(e,e,n)$, for fixed $e$ and increasing $n$ (see \cite[Thm.~8]{Brieskorn73} for the case $e=2$ and \cite{CalSal20} for the general case);
\item homological stability holds for the sequence of complex braid groups of type $B(2e,e,n)$, for fixed $e$ and increasing $n$ (see \cite[Thm.~8]{Brieskorn73} for the case $e=1$ and \cite{CalMar11} for the general case).
\end{itemize}
For $n=0$ we use the convention that $\Conf_n(e)\cong\Conf_n^*(e)$ is the disjoint union of $e$ points, corresponding to letting $y$ be one of the $e$\textsuperscript{th} roots of 1.

As we will see in Proposition \ref{prop:DdmoduleoverA}, the space $\A$ has a structure of $E_1$-algebra\footnote{In fact, this even extends to an $E_2$-algebra structure, which can further be identified with the free $E_2$-algebra generated by a point; in this article we shall only need the $E_1$-algebra structure on $A$.}, whereas for all $e\ge1$ the space $\D(e)$ has a structure of right $E_1$-module over $\A$. In both cases, multiplication is given by juxtaposition of configurations of points in $\bC$.

\subsection{Localization of monoids}

Recall that for a monoid $A$ the group completion $G(A)$ also inherits the structure of a monoid,  which is well-defined up to a contractible space of choices. 
\begin{lemma}
\label{lem:localization_via_fiber}
Let $A$ be a monoid in spaces and let $M$ be a right $A$-module. Then there is a chain of equivalences of spaces
\[
L_A(M)\simeq \bB(M,A,G(A))\simeq \fib(\bB(M,A,*)\to \bB A).
\]
\end{lemma}
\begin{proof}
Let $\Mod^{\mathrm{loc}}_{A}$ be the full subcategory of $\Mod_{A}$ spanned by local objects. Note first that we have an equivalence $\Mod^{\mathrm{loc}}_{A}\simeq \Mod_{G(A)}$, by definition of $G(A)$. Moreover, the full subcategory inclusion can be identified with the restriction  functor $\Mod_{G(A)}\to \Mod_{A}$. Passing to left adjoints, the localization functor $L_A\colon\Mod_A \to\Mod^{\mathrm{loc}}_A$ can be identified with the relative tensor product $-\otimes_{A}G(A)\colon\Mod_A\to\Mod_{G(A)}$, which can be concretely modelled by the two-sided bar construction~$\bB(-,A,G(A))$ (see e.g.~\cite[Proposition~4.6.2.17]{lurie2017higher} 
or \cite[Lemma~2.2.6]{francis2008derived}). The first equivalence follows. 

For the second equivalence, we first recall that 
May proves that for a group-like $G$ and a $G$-module $M$, we have an equivalence $M\simeq \fib(\bB(M,G,*)\to\bB G)$ as $G$-modules (see~\cite[Theorem 7.6 and Proposition 7.9]{May75}). Applying this to $M=\bB(M,A,G(A))$ and $G=G(A)$ we get
\[
\bB(M,A,G(A))\simeq\fib(\bB(\bB(M,A,G(A)),G(A),*)\to\bB G(A)).
\]
Now, associativity of the bar construction yields
\[
\bB(\bB(M,A,G(A)),G(A),*)\simeq\bB(M,A,\bB(G(A),G(A),*))\simeq \bB(M,A,*),
\]
and the equivalence $\bB G(A)\simeq\bB A$ allows us to write
\[
\bB(M,A,G(A))\simeq\fib(\bB(M,A,*)\to\bB A),
\]
concluding the chain of equivalences in the statement.
\end{proof}

To connect $L_A(M)\simeq \bB(M,A,G(A))$ with the mapping telescope we are actually interested in -- in that it contains information about the stable homology of components of $M$ -- we will also use the following lemmas.
\begin{lemma}
\label{lem:localization_quillenization}
Let $A$ be a monoid with $\pi_0(A)\cong\bN$, and let $a$ be a point in the component $A_1$ of $A$ corresponding to $1\in\bN$. Assume that the group-completion Theorem \cite{McDuffSegal} applies to $A$, and in particular that $G(A)$ is equivalent, as a space, to the quillenization of the mapping telescope
\[
\colim(A\xrightarrow{\_\cdot a}A\xrightarrow{\_\cdot a}A\xrightarrow{\_\cdot a}\dots).
\]
Then the canonical map 
\[
\colim(M\xrightarrow{\_\cdot a}M\xrightarrow{\_\cdot a}M\xrightarrow{\_\cdot a}\dots)\longrightarrow L_A(M)
\]
induces an equivalence under quillenization.
\end{lemma}
\begin{proof}
Denoting by $\Tel_a(-)$ the mapping telescope of a right $A$-module by the iterated action of $a$, and using the first equivalence in Lemma~\ref{lem:localization_via_fiber}, we have a chain of equivalences
\[
\begin{split}
L_A(M)^+&\simeq \bB(M,A,G(A))^+\simeq \bB(M,A,\Tel_a(A)^+)^+\overset{\star}{\simeq}\bB(M,A,\Tel_a(A))^+\\
&\overset{\star\star}{\simeq}\Tel_a(\bB(M,A,A))^+\simeq\Tel_a(M)^+.
\end{split}
\]
In particular, the equivalence $\star$ follows from the fact that quillenization preserves homotopy colimits and finite products, so that both sides agree with the colimit for $[p]\in\Delta^{\mathrm{op}}$ of $M^+\times (A^+)^p\times\Tel_a(A)^+$. The equivalence $\star\star$ follows from exchanging homotopy colimits over $\bN_{\le}$ and $\Delta^{\mathrm{op}}$.
\end{proof}

\begin{lemma}\label{lem:factorization}
Let $A$ be a monoid in spaces and let $M$ be a right $A$-module. Assume that there exists a map $M\to A$ of right $A$-modules. Then there is an equivalence $L_A(M)\simeq G(A)\times \bB(M,A,*)$ of right $G(A)$-modules.
\end{lemma}
\begin{proof}
The hypothesis implies that the map $M\to*$ of right $A$-modules factors as a composite $M\to A\to *$. We may therefore compute a chain of equivalences of right $G(A)$-modules as follows:
\[
\begin{split}
L_A(M)&\simeq\fib(\bB(M,A,*)\to \bB(*,A,*))\\
&\simeq \fib(\bB(M,A,*)\to \bB(A,A,*)\to \bB(*,A,*))\\
&\simeq\fib(\bB(M,A,*)\to *\to \bB(*,A,*))\\
&\simeq\Omega \bB A\times \bB(M,A,*)\simeq G(A)\times \bB(M,A,*),
\end{split}
\]
concluding the proof.
\end{proof}

\section{Setting the module structures}\label{sec:setting_modules}
In this section we show that $\A:=\coprod_{n\ge0}\Conf_n$ and $\D^*(e):=\coprod_{n\ge0}\Conf_n^*(e)$ have the structure of  monoids, that $\D(e):=\coprod_{n\ge0}\Conf_n(e)$ is  both an  $\A$-module and a $\D^*(e)$-module, and that these two module structures are compatible. 

\begin{notation}
In this section we write $\bz$ for a generic configuration $\{ z_1, \ldots, z_n\} \in \Conf_n$.
\end{notation}
\begin{notation}
For a subspace $X\subseteq \bC$, we denote by $\Conf_n(X)$ the subspace of $\Conf_n$ comprising configurations $\bz$ such that $z_i\in X$ for all $i$.
\end{notation}
\begin{definition}
Given $t \in \bR_{\geq 0}$, we define the continuous injective map $\tau_t: \bC^* \to \bC^*$:
\[
\tau_t(z) := z \frac{|z|+t}{|z|}.
\]
For every cone $X\subseteq \bC^*$, i.e.~every subset closed under multiplication by $\bR_{>0}$, the map $\tau_t$ induces a self map on $\Conf_n(X)$ given by the formula
\[\bz \mapsto \tau_t(\bz) = \left\{z_1\frac{|z_1|+t}{|z_1|},\dots,z_n\frac{|z_n|+t}{|z_n|}\right\}.\]
\end{definition}
\begin{definition}
\label{defn:Y}
We denote by
\[Y_e\colon\Conf_m(\bR_{>0}\times \mathrm{i}\bR)\to\bC\]
the unique continuous function 
satisfying $Y_e(\{z_1,\dots,z_n\})^e=\prod_{i=1}^nz_i$ and attaining positive real values on $\Conf_n(\bR_{>0})$.
\end{definition}

We  prove the following preliminary proposition.
\begin{proposition}\label{prop:DdmoduleoverA}
There is a monoid structure on $\A$ and a right $\A$-module structure on $\D(e)$.
\end{proposition}
\begin{proof}
We replace $\Conf_n$ and $\Conf_n(e)$ by their ``Moore'' versions
\begin{align*}
\CConf_n&:=\left\{(\bz,t)\in\Conf_n\times\bR_{\ge0}\ |\ \forall i\ \Re(z_i)>0, |z_i| < t\right\}\simeq\Conf_n;\\
\CConf_n(e)&:=\left\{(\bz,y,t)\in\Conf_n(e)\times\bR_{\ge0}\ \left|\begin{array}{l}
\forall i \ |z_i| < t   \end{array} \right. \right\}\simeq\Conf_n(e). 
\end{align*}
We may now regard $\AA:=\coprod_{n\ge0}\CConf_n$ as a strictly associative and strictly unital topological monoid.  We do this by using the multiplication operation which is described as follows: 
\begin{align*}
\CConf_n \times \CConf_m & \to\  \CConf_{n+m},\\
(\bz,t),(\bz',t') &\mapsto\ \left(\bz \cup \tau_t(\bz'),t+t'\right).
\end{align*}

We may similarly turn $\DD(e):=\coprod_{n\ge0}\CConf_n(e)$ into a strictly associative right module over $\AA$. The formula for multiplication is similar and uses the function $Y_e$ from Definition \ref{defn:Y}:
\begin{align*}
\CConf_n(e) \times \CConf_m & \to\  \CConf_{n+m}(e),\\ 
(\bz,y,t),(\bz',t')&\mapsto\left(\bz \cup \tau_t(\bz'),y\cdot Y_e(\tau_t(\bz')),t+t'\right).
\end{align*}
\end{proof}
\begin{remark}
\label{remark:cospan}
We observe that there is a cospan of maps of (strict) right $\AA$-modules
\[
\DD(e)\to\DD(1)\stackrel{\simeq}{\hookleftarrow}\AA.
\]
The first map forgets $y$, the second is the natural inclusion, which is an  equivalence. We thus obtain a map of right $\A$-modules $\D(e)\to\A$ in $\Mod_\A$. 

Applying Lemma \ref{lem:factorization} we obtain an equivalence of right $G(\A)$-modules
\[
L_\A(\D(e)) \simeq G(\A) \times \bB(\D(e),\A,*).
\]
\end{remark}




\begin{definition} \label{def:HMapOfMonoids}
We denote by $ H\colon\AA\to\DD^{*}(e)$ the map of strict monoids given by 
\begin{align*}
\CConf_n &\to \CConf^{*}_{n}(e),\\
(\bz,t) & \mapsto(\bz,Y_e(\bz),t),
\end{align*}
where $Y_e$ is the function from Definition \ref{defn:Y}.
\end{definition}

The $\A$-module structure on $\D(e)$ extends to a $\D^*(e)$-module structure, given informally by concentric concatenation of configurations, in the following sense.

\begin{proposition}
\label{prop:D*}
There is a monoid structure on $\D^*(e)$ for which the following hold:
\begin{enumerate}
\item there is a $\D^*(e)$-module structure on $\D(e)$; 

\item 
the $\A$-module structure on $\D(e)$ obtained by restriction along $H\colon\A\to\D^*(e)$ agrees with the  $\A$-module structure on $\D(e)$ defined in Proposition \ref{prop:DdmoduleoverA}.
\end{enumerate}
\end{proposition}

\begin{proof} 
Consider the following strict model for $\Conf_n^{*}(e)$:
\[
\CConf^{*}_n(e) = \left\{(\bz,y,t) \in \Conf^*_n(e) \times\bR_{\ge0}\ \left|\begin{array}{l}
\forall i \   |z_i| <  t\end{array} \right. \right\}\simeq\Conf_n^*(e).
\]
We regard $\DD^{*}(e):=\coprod_{n\ge0}\CConf^{*}_n(e)$ as a strict topological monoid, with multiplication given by
\begin{align*}
\CConf^{*}_n(e) \times \CConf^{*}_m(e) &\to \CConf^{*}_{m+n}(e),\\
(\bz,y,t), (\bz',y',t') &\mapsto \left(\bz \cup \tau_t(\bz'), yy' \prod_{j=1}^m\sqrt[e]{\frac{|z_j'|+t}{|z_j'|}},t+t'\right).
\end{align*}
We may regard $\DD^{*}(e):=\coprod_{n\ge0}\CConf^{*}_n(e)$ as a strict right module over $\AA$ by restricting the $\AA$-module structure of $\DD(e)$ to $\DD^{*}(e)$.




\begin{enumerate}
    \item 
The $\DD^*(e)$-module structure on $\DD(e)$ is given by concentric concatenation, extending the multiplicative structure of $\DD^*(e)$:

\begin{align*}
\CConf_n(e) \times \CConf^{*}_m(e) &\to \CConf^{*}_{m+n}(e),\\
(\bz,y,t), (\bz',y',t') &\mapsto \left(\bz \cup \tau_t(\bz'), yy' \prod_{j=1}^m\sqrt[e]{\frac{|z_j'|+t}{|z_j'|}},t+t'\right).
\end{align*}

\item  This follows from a direct inspection of the definitions.
\end{enumerate}
\end{proof}

\begin{remark}
\label{rem:pushout_presentation_preliminaries}
For the next proposition, we note that the cyclic group $\bZ/e$ acts freely on the space $\D^*(e)$ by multiplying the coordinate $y$ of a configuration by the root of unity $\exp(2\pi i/e)$. The action is not one by automorphisms of monoids, but it is one by automorphisms of $\D^*(e)$-modules. This is witnessed by the action of $\bZ/e$ on $\DD^*(e)$ also given by multiplying the coordinate $y$ by $\exp(2\pi i/e)$, which is an action by strict $\DD^*(e)$-module automorphisms. By restriction, also $\bZ$ acts on $\D^*(e)$ by $\D^*(e)$-module automorphisms.

We further observe that if $A$ is a monoid and $\bZ$ acts on $A$ by $A$-module automorphisms, then the homotopy quotient $A\sslash\bZ$ agrees with the coequalizer of the diagram $A\rightrightarrows A$ in $\Mod_A$, where the two maps are given by the identity of $A$ and by multiplication by $1\in\bZ$, respectively. Since $A$ is freely generated by a point, we obtain that for any $A$-module $M$, a map of $A$-modules $A\sslash\bZ\to M$ is uniquely determined by a choice of a point $m\in M$ and a path from $m$ to $1\odot m$ in~$M$.
\end{remark}

\begin{proposition}
\label{prop:pushoutpresentation}
There is a homotopy pushout square in $\Mod_{\D^*(e)}$ as follows, in which all quotients are meant as homotopy quotients:
\[
\begin{tikzcd}[row sep=10pt]
\D^*(e)\sslash\bZ\ar[r]\ar[d]\ar[dr,phantom,"\ulcorner"very near end]&\D^*(e)\sslash(\bZ/e)\ar[d]\\
\D^*(e)\ar[r]&\D(e).
\end{tikzcd}
\]
The top horizontal map is induced by the surjective group homomorphism $\bZ\twoheadrightarrow\bZ/e$.
The left vertical map is induced by the path of configurations of in $\D^*(e)$ starting at $(\{z_1=1\},y=1)$ and ending at 
\[
\big(\{z_1=1\},y=\exp(2\pi i/e)\big)\ =\ 1\odot\big(\{z_1=1\},y=1\big),
\]
the path comprising configurations of one point $z_1$ spinning once counterclockwise around 0.
\end{proposition}
Before giving the proof, we observe that the homotopy quotient $\D^*(e)\sslash(\bZ/e)$ may be modelled by the actual quotient of topological spaces $\D^*(e)/(\bZ/e)$, since the action of $\bZ/e$ on $\D^*(e)$ is free.
\begin{proof}[Proof of Proposition \ref{prop:pushoutpresentation}]
We decompose $\DD(e)$ as a union of two open subspaces, each carrying a strict $\DD^{*}(e)$-module structure by restriction. The first space comprises configurations $(\bz,y,t)$ with $\bz$ avoiding $0$, and this is isomorphic to $\DD^{*}(e)$ as a right module over $\DD^{*}(e)$, hence equivalent to the bottom left corner $\D^*(e)$. The second space comprises configurations $(\bz,y,t)$ satisfying the following: there exists $1\le i\le n$ such that $|z_i|<|z_j|/2$ for all $j\neq i$ (in particular $n\ge1$). This second space deformation retracts by linear interpolation of $z_i$ onto the subspace of $\DD(e)$ comprising configurations for which there exists a (necessarily unique) $1\le i\le n$ with $z_i=0$; equivalently, configurations with $y=0$. The deformation retraction is one of strict $\DD^{*}(e)$-modules, and the latter space is homeomorphic to $\DD^*(1)\cong\DD^{*}(e)/(\bZ/e)\simeq\D^*(e)\sslash(\bZ/e)$, i.e.~it is equivalent to the top right corner, as a $\DD^*(e)$-module. 

The intersection of the two mentioned open subspaces of $\DD(e)$ comprises all configurations $(\bz,y,t)\in\D(e)$ such that there exists $1\le i\le n$ with $0\neq|z_i|<|z_j|/2$ for all $j\neq i$; this intersection is  equivalent, as a $\DD^{*}(e)$-module, to the homotopy quotient $\D^*(e)\sslash\bZ$. 


In terms of the strict models $\DD(e)$ and $\DD^{*}(e)$ introduced above, for every $n$ we have the following (strict) pushout of topological spaces:
\begin{equation}\label{eq:pushout}
\begin{tikzcd}
\CConf_n^{0*}(e)\ar[r]\ar[d]\ar[dr,phantom,"\ulcorner"very near end]&\CConf_n^{0}(e)\ar[d]\\
\CConf_n^{*}(e)\ar[r]&\CConf_n(e),
\end{tikzcd}
\end{equation}
where \[\CConf_n^{0}(e) = \left\{ (\bz,y, t) \in \CConf_n(e) \left|      \exists\  i \mbox{ such that } 2 |z_i| < \min_{j \neq i }(|z_j|) 
\right.\right\}\] 
and 
\[\CConf_n^{0*}(e)= \left\{ (\bz,y, t) \in \CConf_n(e) \left|      \exists\  i \mbox{ such that } 0\neq 2 |z_i| < \min_{j \neq i }(|z_j|) 
\right.\right\} .\]
We note that taking the pullback of the universal cover
\begin{align*}
    \{ w \in \bC \mid \Re(w) <0 \} = \bC_{\Re <0} & \to D^2\setminus \{0\} = \{z \in \bC \mid 0<|z|<1 \},\\
    w & \mapsto \exp(w)
\end{align*}
with respect to the map 
\begin{align*}
    \CConf_n^{0*} & \to D^2\setminus\{ 0 \},\\
    (\bz, y,  t) & \mapsto \frac{2 z_i}{\min_{j \neq i }|z_j|}
\end{align*}
we obtain the $\bZ$-covering
\[\widetilde{\CConf}{}_{n-1}^{0*}(e)  \to  \CConf_n^{0*}(e);\]
here we are using the notation
\[ \widetilde{\CConf}{}_{n-1}^{0*}(e) = \left\{ (w, \bz, y, t) \in \bC_{\Re<0} \times \Conf_{n-1}^*\times \bC^* \times \bR_{\geq 0} \left| \begin{array}{l}
\epsilon \exp(w) \prod_j z_j= y^e,\\
\forall i \   |z_i| <  t
\end{array}\right.\right\} \]
with $\epsilon = \frac{\min |z_j|}{2}$ for $n>1$ and $\epsilon = \frac{t}{2}$ if $n=1$, 
and the covering  \hypertarget{covering}{} is given by
\begin{align*} 
   \widetilde{\CConf}{}_{n-1}^{0*}(e) & \to  \CConf_n^{0*}(e), \\
(w,\bz, y,  t) & \mapsto \left(\left\{\epsilon \exp(w) \right\} \cup \bz, y ,t\right).
\end{align*}
As a consequence, $\widetilde{\DD}^0(e) := \coprod_{n\geq0 }\widetilde{\CConf}{}_{n-1}^{0*}(e)$ is a $\DD^*(e)$-module and the map  $\widetilde{\DD}^0(e) \to \DD^0(e)$ is a map of $\DD^*(e)$-modules.

The space $\widetilde{\CConf}{}_{n-1}^{0*}(e)$ deformation retracts to the subspace given by the condition $w=-1$, that is onto a subspace homeomorphic to $\CConf_{n-1}^*(e)$, by the linear interpolation
\[ 
(r,(w,\bz, y, t)) \mapsto (-r+(1-r)w, \bz, \exp(r(-1-w)/e)y, t).
\]
The union over $n$ of the above retractions gives a retraction of $\widetilde{\DD}^0(e)$ onto its subspace $\widetilde{\DD}^0(e)_{w=-1}$ given by imposing the condition $w=-1$. We observe that $\widetilde{\DD}^0(e)_{w=-1}$ is a strict $\DD^*(e)$-submodule of $\widetilde{\DD}^0(e)$ which is isomorphic to $\DD^*(e)$ as a strict $\DD^*(e)$-module; moreover, the above retraction of $\widetilde{\DD}^0(e)$ onto $\widetilde{\DD}^0(e)_{w=-1}$ is one of strict $\DD^*(e)$-modules.

The action of $1 \in \bZ$ on $\widetilde{\CConf}{}_{n-1}^{0*}(e)$ maps 
\[ 
(w,\bz, y,  t) \mapsto (w+2 \pi i, \bz, y\exp(2\pi i/e),  t);
\]
 hence, if we compose with the retraction on $\CConf_{n-1}^*(e)$ we obtain that the action of $1 \in \bZ$ on $\CConf_{n-1}^*(e)$ is again
\[
(\bz, y,  t) \mapsto (\bz, y\exp({2 \pi i /e}),  t).
\]
If we consider the space $\DD^{0*}(e) = \coprod_{n} \CConf_n^{0*}(e)$ with the natural structure of $\DD^{*}(e)$-module, we get the following equivalence of $\DD^{*}(e)$-modules, concluding the proof:
\[
\DD^{0*}(e) \simeq \DD^{*}(e)\sslash\bZ.
\]
\end{proof}

\section{Proofs of the main theorems}

In view of Lemma~\ref{lem:factorization}, to prove Theorem \ref{thm:A} it is sufficient to identify $\bB(\D(e),\A,*)$  with the space $\fib(S^3\xrightarrow{e}S^3)$. Similarly, Theorem \ref{thm:B} reduces to the equivalence $\bB(\D^*(e),\A,*)\simeq\Omega S^2\left<e\right>$. Before giving the proof of Theorems \ref{thm:A} and \ref{thm:B}, we shall collect some preliminary technical results.

\begin{lemma}
\label{lem:pullback_BD*e}
Consider the pullback square of groups on the left; then the induced commutative square of classifying spaces on the right is again a pullback square:
\[
\begin{tikzcd}
\D^{*}(e)\ar[r]\ar[d]\ar[dr,phantom,"\lrcorner"very near start]&\bC^*\ar[d,"(-)^e"]\\
\D^{*}(1)\ar[r,"\prod"]&\bC^*;
\end{tikzcd}
\quad\quad
\begin{tikzcd}
\bB\D^{*}(e)\ar[r]\ar[d]\ar[dr,phantom,"\lrcorner"very near start]&\bB\bC^*\ar[d,"\bB(-)^e"]\\
\bB\D^{*}(1)\ar[r,"\prod"]&\bB\bC^*.
\end{tikzcd}
\]
\end{lemma}
\begin{proof}
By definition, $\bB\D^{*}(e)$  is the geometric realization of the simplicial space $\D^{*}(e)^{\times\bullet}\colon\Delta^\op\to\cS$ given by
\[
[p]\mapsto\D^{*}(e)^p\simeq\D^{*}(1)^p\times_{(\bC^*)^p}(\bC^*)^p.
\]
We have in fact a pullback square of simplicial spaces
\[
\begin{tikzcd}
\D^{*}(e)^{\times\bullet}\ar[r]\ar[d]\ar[dr,phantom,"\lrcorner"very near start]&(\bC^*)^{\times\bullet}\ar[d]\\
\D^{*}(1)^{\times\bullet}\ar[r]&(\bC^*)^{\times\bullet}.
\end{tikzcd}
\]
Note that all terms are  (non-complete) Segal spaces. Consider now the space $\bB\bC^*$ as a constant simplicial space and note that there is a canonical map of simplicial spaces $(\bC^*)^{\times\bullet}\to\bB\bC^*$. Let $X_\bullet$ denote the fiber product of the cospan of simplicial spaces $\D^{*}(1)^{\times\bullet}\to \bB\bC^*\leftarrow (\bC^*)^{\times\bullet}$ obtained from the composite maps $\D^{*}(1)^{\times\bullet}\to(\bC^*)^{\times\bullet}\to\bB\bC^*$ and $(\bC^*)^{\times\bullet}\xrightarrow{(-)^e}(\bC^*)^{\times\bullet}\to\bB\bC^*$. 
Then there is a canonical map of Segal spaces $\D^{*}(e)^{\times\bullet}\to X_\bullet$, which is easily seen to be a Dwyer-Kan equivalence:
it is essentially surjective because $X_0\simeq\bC^*$ is connected, and it is fully faithful because it is the pullback of three fully faithful maps of Segal spaces. It follows that $\bB\D^{*}(e)\xrightarrow{\simeq}|X_\bullet|$, as Dwyer-Kan equivalences of Segal spaces induce equivalences on geometric realizations. Finally, we have that $X_\bullet$ is a pullback of simplicial spaces in which the middle simplicial space is constant, so that $|X_\bullet|\xrightarrow{\simeq}\bB\D^{*}(1)\times_{\bB\bC^*}\bB\bC^*$.
\end{proof}
\begin{remark}
\label{rem:pullback_BD*e}
Taking loop spaces in the right pullback square from Lemma \ref{lem:pullback_BD*e}, we obtain in particular a pullback square of groups
\[
\begin{tikzcd}
G(\D^{*}(e))\ar[r]\ar[d]\ar[dr,phantom,"\lrcorner"very near start]&\bC^*\ar[d,"(-)^e"]\\
G(\D^{*}(1))\ar[r,"\prod"]&\bC^*
\end{tikzcd}
\]
whose vertical fibers are equivalent to $\bZ/e$. This shows that $G(\D^*(e))$, as a space, is an $e$-fold covering of $G(\D^*(1))$.
\end{remark}

\begin{remark}
\label{rem:McDuff}
The canonical map of $\A$-modules $L_\A(\D^*(1))\to G(\D^*(1))$ fits into the following commutative diagram, whose rows are fiber sequences:
\[
\begin{tikzcd}
L_\A(\D^*(1)) \ar[r] \ar[d] &\bB(\D^*(1),\A,*) \ar[r] \ar[d] & \bB(*,\A,*) \ar[d]\\
G(\D^*(1)) \ar[r] &\bB(\D^*(1),\D^*(1),*)\ar[r]&\bB(*,\D^*(1),*).
\end{tikzcd}
\]
We next argue that the right square in the previous diagram is a pullback square, and hence, that the left vertical map is an equivalence. To see this, consider the annulus $M=S^1\times[0,1]$, and decompose $S^1=S^1_+\cup S^1_-$ as a union of two intervals. 
Then by \cite[Theorem 2.6]{McDuff} one may identify
the right commutative square above with the following one
\[
\begin{tikzcd}
\map((M,S^1\times0\sqcup S^1_+\times1);(S^2,*))\ar[r]\ar[d]&\map((M,S^1_+\times1);(S^2,*))\simeq\Omega S^2\ar[d]\\
*\ar[r]&\map((M,\emptyset);(S^2,*))\simeq\cL S^2,
\end{tikzcd}
\]
and the latter is evidently a pullback square. Here we are considering spaces of continuous maps sending the relevant subspace of $\partial M$ constantly to $*\in S^2$. We are also denoting by $\cL S^2$ the free loop space of $S^2$.
It follows that $L_\A(\D^*(1))\xrightarrow{\simeq}G(\D^*(1))$; the latter may as well be identified with~$\Omega(\cL S^2)$.
\end{remark}
\begin{proposition} \label{prop:fibration}
There is a fiber sequence of grouplike monoids
\[
G(\A)\to G(\D^*(e))\to\Omega S^2\left<e\right>,
\]
where $\Omega S^2\left<e\right>$ denotes the connected $e$-fold covering of the connected grouplike monoid $\Omega S^2$, using that the latter satisfies $\pi_1(\Omega S^2)\cong\bZ$.
\end{proposition}
\begin{proof}
In the case $e=1$, scanning \cite{Segal73,McDuff} provides the following commutative diagram of groups
\[
\begin{tikzcd}
G(\A)\ar[r]\ar[d,"\simeq"]&G(\D^{*} (1))\ar[d,"\simeq"]\\
\Omega^2S^2\ar[r]&\Omega(\cL S^2),
\end{tikzcd}
\]
where the bottom map is induced by the inclusion $\Omega S^2\to\cL S^2$. As the latter agrees with the inclusion of the fiber of the evaluation map $\cL S^2\to S^2$, the result follows for $e=1$ from the fiber sequence $\Omega^2S^2\to \Omega(\cL S^2)\to\Omega S^2$.

We further observe that the map $\prod\colon\D^{*}(1)\to\bC^*$ agrees as a map of monoids with the composite map $\D^{*}(1)\to L(\D^{*}(1))\simeq\Omega(\cL S^2)\to\Omega S^2\to\bC^*$, where the last map is given by 1-truncation, using that $\Omega S^2$ is connected with fundamental group $\bZ$. From the equivalence $\Omega S^2\left<e\right>\simeq\Omega S^2\times_{\bC^*}\bC^*$, and using Lemma \ref{lem:pullback_BD*e} and Remark \ref{rem:pullback_BD*e} it then follows that we have for generic $e$ a commutative diagram of groups
\[
\begin{tikzcd}
G(\A)\ar[d,equal]\ar[r]&G(\D^{*}(e))\ar[d]\ar[r]\ar[dr,phantom,"\lrcorner"very near start]&\Omega S^2\left<e\right>\ar[d]\ar[r]\ar[dr,phantom,"\lrcorner"very near start]&\bC^*\ar[d,"(-)^e"]\\
G(\A)\ar[r]&G(\D^{*}(1))\ar[r]&\Omega S^2\ar[r]&\bC^*,
\end{tikzcd}
\]
reducing the case of generic $e$ to the case $e=1$.
\end{proof}

\begin{corollary}\label{cor:D*}
The canonical map of $\A$-modules $  L_\A(\D^*(e))\to G(\D^*(e))$ is an equivalence.
\end{corollary}
\begin{proof} 
For $e=1$ we have observed that the canonical map of $\A$-modules $L_\A(\D^*(e))\to G(\D^*(e))$ is an equivalence in Remark \ref{rem:McDuff}.

For generic $e$ we have a pullback square of simplicial spaces below on the left, giving a pullback square of spaces on the right under geometric realization, using that the bottom right simplicial space in the pullback on the left is constant:
\[
\begin{tikzcd}
\D^*(e)\times\A^{\times\bullet}\times G(\A)\ar[r]\ar[d]\ar[dr,phantom,"\lrcorner"very near start]&\bC^*\ar[d,"(-)^e"]\\
\D^*(1)\times\A^{\times\bullet}\times G(\A)\ar[r]&\bC^*;
\end{tikzcd}\quad\quad
\begin{tikzcd}
L_\A(\D^*(e))\ar[r]\ar[d]\ar[dr,phantom,"\lrcorner"very near start]&\bC^*\ar[d,"(-)^e"]\\
L_\A(\D^*(1))\ar[r]&\bC^*.
\end{tikzcd}
\]
Comparing this with the pullback square of groups from Remark \ref{rem:pullback_BD*e}, we obtain that the following square is also a pullback square:
\[
\begin{tikzcd}
L_\A(\D^*(e))\ar[r]\ar[d]\ar[dr,phantom,"\lrcorner"very near start]&G(\D^*(e))\ar[d]\\
L_\A(\D^*(1))\ar[r,"\simeq"]&G(\D^*(1)).
\end{tikzcd},
\]
and since the bottom map is an equivalence, we conclude that the top map is an equivalence too.
\end{proof}
\begin{proof}[Proof of Theorem \ref{thm:B}]

The composite map of $\A$-modules $\D^*(e)\to\D^*(1)\to\D(1)\to\A$, together with Lemma \ref{lem:factorization}, allows us to write an equivalence of $\A$-modules $L_\A(\D^*(e))\simeq G(\A)\times\bB(\D^*(e),\A,*)$.

For $e=1$ the space $\bB(\D^*(1),\A,*)$ is equivalent to $\Omega S^2$, as observed in Remark \ref{rem:McDuff}. For generic $e$, the pullback square of simplicial spaces on the left in the following gives rise to the pullback square of spaces on the right, again using that the bottom right simplicial space is constant:
\[
\begin{tikzcd}
\D^*(e)\times\A^{\times\bullet}\ar[r]\ar[d]\ar[dr,phantom,"\lrcorner"very near start]&\Omega S^2\left<e\right>\ar[d]\\
\D^*(1)\times\A^{\times\bullet}\ar[r]&\Omega S^2;
\end{tikzcd}\quad\quad
\begin{tikzcd}
\bB(\D^*(e)\A,*)\ar[r]\ar[d]\ar[dr,phantom,"\lrcorner"very near start]&\Omega S^2\left<e\right>\ar[d]\\
\bB(\D^*(1),\A,*)\ar[r,"\simeq"]&\Omega S^2.
\end{tikzcd}
\]
Using that the bottom map is an equivalence of spaces, we conclude that the top map is an equivalence of spaces too.
\end{proof}

\begin{proof}[Proof of Theorem \ref{thm:A}]
By Proposition \ref{prop:pushoutpresentation} there is a homotopy pushout square of spaces as follows
\[
\begin{tikzcd}
\bB(\D^*(e)/\bZ,\A,*)\ar[r]\ar[d]\ar[dr,phantom,"\ulcorner"very near end]&\bB(\D^*(e)/\bZ/e,\A,*)\ar[d]\\
\bB(\D^*(e),\A,*)\ar[r]&\bB(\D(e),\A,*).
\end{tikzcd}
\]
In fact the spaces involved in the previous pushout square have a residual $\Omega S^2\left<e\right>$-action by Proposition \ref{prop:D*}, and the previous is a pushout of $\Omega S^2\left<e\right>$-modules.
Using that bar constructions commute with homotopy quotients we
may rewrite the previous as follows:
\[
\begin{tikzcd}
(\Omega S^2\left<e\right>)/\bZ\ar[d]\ar[r]\ar[dr,phantom,"\ulcorner"very near end]&(\Omega S^2\left<e\right>)/\bZ/e\ar[d]\\
\Omega S^2\left<e\right>\ar[r]&B(\D(e),\A,*).
\end{tikzcd}
\]
Using that $\Omega S^2\left<e\right>$ is a grouplike monoid, we may convert each term in the previous pushout square of $\Omega S^2\left<e\right>$-modules into the fiber of a map towards $\bB(\Omega S^2\left<e\right>)$. All in all we obtain an equivalence
\[
\bB(\D(e),\A,*)\simeq\fib(\bB\bZ/e\sqcup_{\bB\bZ}*\longrightarrow \bB(\Omega S^2\left<e\right>)).
\]
It is left to identify the map $\bB\bZ/e\sqcup_{\bB\bZ}*\longrightarrow \bB(\Omega S^2\left<e\right>)$, or at least its fiber. We first observe that the given map must induce an equivalence on $\pi_2$: indeed both the source and the target have $\pi_2$ isomorphic to $\bZ$, and the map must be a surjection (and hence a bijection) because we know that $\pi_1$ of the fiber, that is $\pi_1(\bB(\D(e),\A,*))$, vanishes. This can for instance be deduced from the fact that the stable $H_1(\Conf_n(e))$ is equal to $\bZ$ and is not larger than that.

We may pass to 2-connected covers and write an equivalence
\[
\bB(\D(e),\A,*)\simeq\fib(S^3\to S^3)
\]
for a suitable self map of $S^3$. Here we use that the 2-connected cover of $\bB\bZ/e\sqcup_{\bB\bZ}*$ has the form $S^1\sqcup_{S^1\times S^1}S^1$ and must therefore be $S^3$ if it is 2-connected. On the other hand the 2-connected cover of $\bB(\Omega S^2\left<e\right>)$ is the classifying space of the universal cover of $\Omega S^2\left<e\right>$, which is the classifying space of $\Omega S^3$, which is $S^3$.

To conclude that the degree of the self map of $S^3$ must be precisely~$e$ it suffices to compute $H_2(\bB(\D(e),\A,*))\cong\pi_2(\bB(\D(e),\A,*))$ to be equal to $\bZ/e$. This can for instance be done by computing the stable integral homology group $H_2(\Conf_n(e)) = \bZ/e \times \bZ/2$, as computed in \cite[Thm.~6.4]{CalMar11}.
\end{proof}
As a corollary of Theorems \ref{thm:A} and \ref{thm:B},  we get the following corollary.
\begin{corollary}
The spaces $L_\A(\D(e))$ and $L_\A(\D^*(e))$ agree with their own quillenizations.
\end{corollary}
\begin{proof}
It follows from Theorem~\ref{thm:A} that $\pi_1(L_\A(\D(e)))$ is abelian, whence the claim follows. Similarly, we get the statement for $L_\A(\D^*(e))$ from Theorem \ref{thm:B}.
\end{proof}

\section{A conjectural scanning map}
In this section we introduce an explicit scanning map $s:\mathcal{D}(e) \to \mathrm{fib}(S^3 \stackrel{e}{\rightarrow} S^2)$, that conjecturally is homotopic to the composition of maps of spaces
\[
\D(e)\to L_\A(\D(e))\simeq \Omega^2 S^2 \times \fib(S^3\xrightarrow{e}S^3)\xrightarrow{\text{project}} \fib(S^3\xrightarrow{e}S^3),
\]
where the middle equivalence is provided by 
Theorem \ref{thm:A}. 
We observe that the analogous composite involving the projection onto the first factor $\Omega^2S^2$ is simply the composition of the forgetful map
$\mathcal{D}(e) \to \mathcal{D}(1)=\mathcal{A}$ followed by the Segal scanning map 
$\mathcal{A} \to \Omega^2S^2$  described in \cite{Segal73}.

We model $S^3$ by the space $\bC^2\setminus\{0\}$, and model the map $S^3\stackrel{e}{\to}S^3$ by the map $q\colon\bC^2\setminus\{0\}\to\bC^2\setminus\{0\}$ given by $(w_1,w_2)\mapsto(w_1^e,w_2)$. For fixed $n\ge0$, consider the map of spaces
\[
p_n\colon \Conf_n(e)\to\bC^2-0,\quad p(\bz,y)=\left(y,\sum_{i=1}^n\prod_{j\neq i}z_j\right)=\left(y,\frac d{dz}\prod_{i=1}^n(z+z_i)\Big|_{z=0}\right)
\]
We observe that the above formula indeed gives a point in $\bC^2\setminus\{0\}$, using that the points $z_1,\dots,z_n\in\bC$ are pairwise distinct.

The composite $q\circ p_n\colon\Conf_n(e)\to\bC^2\setminus\{0\}$ is given by the formula
\[
q\circ p_n(\bz,y)=\left(\prod_{i=1}^nz_i,\sum_{i=1}^n\prod_{j\neq i}z_j\right),
\]
and the map
\[
h_n\colon\Conf_n(e)\times[0,1]\to\bC^2\setminus\{0\},\quad(\bz,y,r)\mapsto \left(\prod_{i=1}^n(1-r+rz_i),\sum_{i=1}^n\prod_{j\neq i}(1-r+rz_j)\right)
\]
provides a homotopy between $q\circ p_n$ and the constant map $\Conf_n(e)\to\bC^2\setminus\{0\}$ with value $(1,n)$.

The combination of $p_n$ and the nullhomotopy of $q\circ p_n$ provides the map $s$ on the subspace $\Conf_n(e)\subset\D(e)$.

We may use the above formulas also to describe a conjectural scanning map $s'\colon\D^*(e)\to\Omega S^2\left<e\right>$. Following Remark \ref{rem:OmegaS2_splits}, we may identify $\Omega S^2\left<e\right>$ with the homotopy fiber of the (non-canonically nullhomotopic) map
\[
q'\colon \bC^*\times\bC\to\bC^2\setminus\{0\},\quad (w_1,w_2)\mapsto(w_1^e,w_2),
\]
i.e. the restriction of $q$ on $\bC^*\times\bC$.
The definition of $\Conf^*_n(e)$ as a subspace of $\Conf_n(e)$ implies that $p_n$ restricts to a map $p'_n\colon\Conf^*_n(e)\to\bC^*\times\bC$. The restriction $h'_n$ of $h_n$ on $\Conf_n^*(e)\times[0,1]$ provides a homotopy between $q'\circ p'_n$ and the constant map $\Conf_n^*(e)\to\bC^2\setminus\{0\}$ at $(1,n)$. We define the map of $s'$ on $\Conf^*_n(e)$ as the map induced by  $p'_n$ and $h'_n$, and conjecture that $s'$ is homotopic to the following composition, whose middle equivalence is provided by Theorem \ref{thm:B}.
\[
\D^*(e)\to L_\A(\D^*(e))\simeq \Omega^2 S^2 \times \Omega S^2\left<e\right>\xrightarrow{\text{project}} \Omega S^2\left<e\right>.
\]

\bibliographystyle{alpha}
\bibliography{biblio}
\end{document}